\documentclass{article}
\usepackage{multirow}
\usepackage{tabularx}
\usepackage{float}
\usepackage{mathtools}
\usepackage{hyperref}
\usepackage{authblk}
\usepackage{amsfonts}
\usepackage{amsmath}
\usepackage{amsthm}
\usepackage{amssymb}
\def\multiset#1#2{\ensuremath{\left(\kern-.3em\left(\genfrac{}{}{0pt}{}{#1}{#2}\right)\kern-.3em\right)}}

\newtheorem{theorem}{Theorem}
\newtheorem{proposition}{Proposition}

\theoremstyle{remark}

\theoremstyle{definition}

\begin{document}
\title{On the distribution of orders of Frobenius action on $\ell$-torsion of abelian surfaces}

\author{Kolesnikov~N.\,S.
}
\author{Novoselov~S.\,A.
}
\affil{Immanuel Kant Baltic Federal University \\\normalsize{\{nikolesnikov1,~snovoselov\}@kantiana.ru}}

\date{
\footnotetext[0]{\textit{The reported study was funded by RFBR according to the research project 18-31-00244.}}
\footnotetext[0]{\textit{A preliminary version of this paper was presented by the authors at SibeCrypt'19~\cite{KolesnikovNovoselov2019}.}}
}

\maketitle

\begin{abstract}
    The computation of the order of Frobenius action on the $\ell$-torsion is a part of Schoof-Elkies-Atkin algorithm for point counting on an elliptic curve $E$ over a finite field $\mathbb{F}_q$. The idea of Schoof's algorithm is to compute the trace of Frobenius $t$ modulo primes $\ell$ and restore it by the Chinese remainder theorem. Atkin's improvement consists of computing the order $r$ of the Frobenius action on $E[\ell]$ and of restricting the number $t \pmod{\ell}$ to enumerate by using the formula $t^2 \equiv q (\zeta + \zeta^{-1})^2 \pmod{\ell}$. Here $\zeta$ is a primitive $r$-th root of unity.
    In this paper, we generalize Atkin's formula to the general case of abelian variety of dimension $g$. 
    Classically, finding of the order $r$ involves expensive computation of modular polynomials. We study the distribution of the Frobenius orders in case of abelian surfaces and $q \equiv 1 \pmod{\ell}$ in order to replace these expensive computations by probabilistic algorithms. 
    
    
\end{abstract}
    
\section{Introduction}
The computation of the Frobenius order and its usage for counting points on elliptic curves is a part of Atkin's contribution to Schoof-Elkies-Atkin (SEA) algorithm (see, \cite{Schoof1995} and \cite[\S17.2.2]{CohenFrey+2005}).

Let $E$ be an elliptic curve defined over a finite field $\mathbb{F}_q$ of characteristic $p$ and let $\ell \neq p$ be a prime number. The Frobenius endomorphism on $E[\ell]$ can be represented as an element of $\operatorname{PGL}_2(\mathbb{F}_\ell)$, a projective general linear group of matrices. If $r$ is the order of Frobenius as element of $\operatorname{PGL}_2(\mathbb{F}_\ell)$, then the Frobenius trace $t$ of elliptic curve satisfies Atkin's formula \cite[Prop.6.2]{Schoof1995}:
\begin{equation}
\label{eq:ell_r_zeta}
t^2 \equiv q (\zeta + \zeta^{-1})^2 \pmod{\ell},
\end{equation}
where $\zeta$ is a primitive $r$-th root of unity. So to determine $t \pmod{\ell}$ in the algorithm, we only need to enumerate primitive roots $\zeta$ instead of enumerating all $\ell$ possible variants.
The computation of $r$ itself in SEA-algorithm is done by using the factorization of modular polynomials.

The formula \eqref{eq:ell_r_zeta} can be generalized to abelian varieties of higher dimension. The order $r$ in this case is defined as the order of Frobenius endomorphism as an element of 
$\operatorname{PGL}_{2g}(\mathbb{F}_\ell)$
(see \S\ref{sec:frob_action}).
If $A$ is an abelian surface over a finite field $\mathbb{F}_q$ and $a_1, a_2$ are coefficients of the characteristic polynomial of Frobenius endomorphism on $A$, that is
\[
\chi_{A,q}(T) = T^{4} + a_1 T^{3} + a_2 T^2 + a_{1} q T + q^2,
\]
then
\begin{equation}
\label{eq:intro_g2_a2_limit}
(a_2 - 2q)^2 = \eta_1 \eta_2 q^2 \pmod{\ell}
\end{equation}
and
\begin{equation}
\label{eq:intro_g2_a1_limit}
a_1^2 = (\sqrt{\eta_1} \pm \sqrt{\eta_2})^2 q \pmod{\ell},
\end{equation}
where $\eta_1 = \zeta_1 + \zeta_1^{-1} + 2$, $\eta_2 = \zeta_2 + \zeta_2^{-1} + 2$, and $\zeta_1, \zeta_2$ are $r$-th roots of unity. This formula appears in \cite{Martindale2017} in a slightly different form and with additional restrictions implying that $\zeta_1, \zeta_2$ are primitive. In \cite[Prop.~3.14]{BGGM+2017} there is a more restrictive formula for vanilla abelian surfaces with real multiplication. In our work, we give explicit formulae for any abelian variety of dimension $g$ with relaxed restrictions on $r$-th roots to make it suitable for general case.
We also provide simplified versions of our formulae for dimensions $2,3$.

Therefore, if we know the order $r$, we can reduce the number of coefficients of characteristic polynomial (mod $\ell$) to enumerate in the genus $2$ generalization of Schoof's algorithm \cite{GaudrySchost2012}.
However, modular polynomials \cite{BrokerLauter2009,GaudrySchost2005} for the case of dimension $g \geq 2$ are too big to be precomputed and the effective tools for computing them modulo $p$ in general case are currently missing.
In this work, we develop a probabilistic approach to point counting and study the distribution of order $r$ in the case of abelian surfaces and $q \equiv 1 \pmod{\ell}$.

\textit{Our contribution.} We give a generalization of Atkin's formula to abelian varieties of any dimension. Our formulae are explicit and can be efficiently computed. These new formulae allows us to limit the number of possibilities for $\chi_{A,q}(T) \pmod{\ell}$ in case when the order of the Frobenius on $A[\ell]$ is known. Our second contribution concerns the distribution of orders of matrices in the symplectic group $\operatorname{Sp}_4(\mathbb{F}_\ell)$ as elements of $\operatorname{PSp}_4(\mathbb{F}_\ell)$, a  projective symplectic group. We obtained closed form expressions for the expected value and variance. Furthermore, we calculated the distribution for first primes $\ell \leq 3571$. We applied these results to obtain a distribution of the Frobenius orders of abelian surfaces over finite field of size $q \equiv 1 \pmod{\ell}$.

\textit{The rest of the paper is organized as follows}.
In Section~\S\ref{sec:frob_action} we briefly give the definition and properties for the matrices of the Frobenius action on $A[\ell]$. Section~\S\ref{sec:Atkin_generalization} contains a generalization of Atkin's formulae to any dimension. 
In Section~\S\ref{sec:conjugacy_classes} we give explicit formulae for the orders (up to a scalar) of conjugate classes in $\operatorname{Sp}_4(\mathbb{F}_\ell)$.
In Section~\S\ref{sec:order}
using the assumption that the Frobenius elements are equidistributed in $\operatorname{Sp}_4(\mathbb{F}_\ell)$, we obtain properties of the Frobenius action distribution: expected order, variance and most common values (modes).
In cryptographic applications we need Jacobians of genus $g = 2$ curves with group size at least $256$ bit. Point counting on such curves using generalization of the Schoof's algorithm requires computations modulo all primes $\ell \leq (9g+3) \log{q}$ \cite{Pila1990}. So in this section, we computed the distribution for $\ell=3,...,3571$ as required for applications.
Section~\S\ref{sec:point_counting} contains experimental results. 

\section{Preliminaries}
In this section we introduce notations that are used further.
\begin{itemize}
	\item A general linear group $\operatorname{GL}_{n}(\mathbb{F}_\ell)$ is a group of non-degenerate $n \times n$ matrices with elements in $\mathbb{F}_\ell$;
	\item A special linear group $\operatorname{SL}_{n}(\mathbb{F}_\ell)$ is a group of $n \times n$ matrices having determinant $\pm 1$ with elements in $\mathbb{F}_\ell$;
	\item A symplectic group $\operatorname{Sp}_{2g}(\mathbb{F}_\ell) := \lbrace M \in \mathbb{F}_\ell^{2g \times 2g} \vert M\Omega M^{Tr}=\Omega \rbrace$, where $\Omega$ is a fixed $2g\times 2g$ nonsingular skew-symmetric matrix; 
	\item A general symplectic group $\operatorname{GSp}_{2g}(\mathbb{F}_\ell) := \lbrace M \in \mathbb{F}_\ell^{2g \times 2g} \vert M\Omega M^{Tr}=c\cdot \Omega \rbrace$, for some $c \in \mathbb{F}_\ell$; 
	\item A projective symplectic group $\operatorname{PSp}_{2g}(\mathbb{F}_\ell)$ is a group $\operatorname{Sp}_{2g}(\mathbb{F}_\ell)$ modulo scalar matrices. 
\end{itemize}

\section{Frobenius action on $A[\ell]$}
\label{sec:frob_action}
Let $A$ be an abelian variety of dimension $g$ over a finite field $\mathbb{F}_q$ of characteristic $p$ and $\ell \neq p$ is a prime. From the work of Tate \cite{Tate1966}, we have
\begin{equation*}
\operatorname{End}_{\mathbb{F}_q}(A) \otimes \mathbb{Z}_\ell \simeq \operatorname{End}_{\operatorname{Gal}(\overline{\mathbb{F}}_q/\mathbb{F}_q)}(T_\ell(A)),
\end{equation*}
where $T_\ell(A)$ is the Tate module of $A$ and $\mathbb{Z}_\ell$ is a ring of $\ell$-adic integers.
Since $T_\ell(A) \simeq (\mathbb{Z}_\ell)^{2g}$, the Frobenius endomorphism on $A$ can be represented by the matrix $F \in \operatorname{GL}_{2g}(\mathbb{Z}_\ell)$. Using Weil pairing, it can be shown (see \cite[p.~358]{Ruck1990}) that $F$ has the following properties:
\begin{enumerate}
    \item $F^T M F = q \cdot M$;
    \item the matrix $M$ is skew-symmetric;  
    \item $\det(M)$ is a unit in $\mathbb{Z}_\ell$.
\end{enumerate}
In other words, $F$ belongs to $\operatorname{GSp}_{2g}(\mathbb{Z}_\ell)$.
The matrix of the action of Frobenius on $A[\ell]$ is defined as $F_\ell := F \pmod{\ell}$. In the case of $q \equiv 1 \pmod{\ell}$ this matrix belongs to symplectic group $\operatorname{Sp}_{2g}(\mathbb{F}_\ell)$.
The orders of groups $\operatorname{Sp}_{2g}(\mathbb{F}_\ell)$ and $\operatorname{PSp}_{2g}(\mathbb{F}_\ell)$ are known \cite[\S1.6.4]{BHRD2013}:
\begin{equation}
\label{eq:Sp_2g_order}
\# \operatorname{Sp}_{2g}(\mathbb{F}_\ell) = \ell^{g^2} \prod_{i=1}^{g} (\ell^{2i} - 1)
\end{equation}
and
\begin{equation}
\label{eq:PSp_2g_order}
\# \operatorname{PSp}_{2g}(\mathbb{F}_\ell) = \ell^{g^2} \frac{\prod_{i=1}^{g} (\ell^{2i} - 1)}{\gcd(2,\ell-1)}.
\end{equation}
In this work, we study the orders of matrices $F_\ell$ as elements of $\operatorname{PSp}_{2g}(\mathbb{F}_\ell)$.
From the introduction we know that in dimension $2$ case these orders satisfy Eqs.~\eqref{eq:intro_g2_a2_limit} and \eqref{eq:intro_g2_a1_limit}. In next section we give equations for any dimension. So this information can be used for generalization of SEA-algorithm to higher dimension. 

\section{Generalization of Atkin's formula}
\label{sec:Atkin_generalization}
Now, we derive explicit formulae that relates the order $r$ of the Frobenius action on $A[\ell]$ and the characteristic polynomial $\chi_{A,q}(T)\pmod{\ell}$ of the Frobenius endomorphism on abelian variety $A$ of dimension $g$. 
These formulae are direct generalization of Atkin's formula for the dimension $1$ case (see Proposition 6.2 in \cite{Schoof1995}) which is used in SEA-algorithm. Our formulae can be used for point counting in higher dimension case. 

Let $\varphi$ be the Frobenius endomorphism on $A$ and let
\begin{equation*}
\chi_{A,q}(T) = T^{2g} + a_1 T^{2g-1} + ... + a_g T^g + a_{g-1} q T + ... + a_1 q^{g-1} T + q^g
\end{equation*}
be the characteristic polynomial of $\varphi$. It is known that we can arrange the roots $\lambda_i$ of this polynomial in such way that $\lambda_i \lambda_{i+g} = q$ for $i$ from $1$ to $g$. So we can write
\begin{equation*}
\chi_{A,q}(T) = \prod_{i=1}^{g}(T-\lambda_i)(T-\frac{q}{\lambda_i}).
\end{equation*}
We can associate \cite[\S4]{Hurt2003} the real Weil polynomial $h_{A,q}(T)$ to the characteristic polynomial $\chi_{A,q}(T)$. This polynomial $h_{A,q}(T)$ has the properties: \[
\chi_{A,q}(T) = T^g h_{A,q}\left(T + \frac{q}{T}\right)
\]
and
\[h_{A,q}(T) = \prod\limits_{i=1}^{g}\left(T-\left(\lambda_i+\frac{q}{\lambda_i}\right)\right).\] Let $h_{A,q}(T) = T^g + b_1 T^{g-1} + ... + b_{g-1} T + b_g$. We can write \cite[p.~4,~Th.9]{Hurt2003}:
\begin{equation}
\label{eq:a_from_b_even_case}
a_{2k} = b_{2k} + \sum\limits_{i=1}^{k}\binom{g-2(k-i)}{i}q^i b_{2(k-i)}
\end{equation}
and
\begin{equation}
\label{eq:a_from_b_odd_case}
a_{2k+1} = b_{2k+1} + \sum\limits_{i=1}^{k}\binom{g-2(k-i)-1}{i}q^i b_{2(k-i)+1}.
\end{equation}
So if we know $h_{A,q}(T)$ then we can easily find $\chi_{A,q}(T)$. There are also recurrent formulae\cite[\S17.1.2]{CohenFrey+2005} for the coefficients $a_k$ in terms of powers of roots which can be obtained via Newton-Girard formulae:
\begin{equation}
\label{eq:a_k_from_power_of_roots}
k a_k = S_k + S_{k-1} a_1 + S_{k-2} a_2 + ... + S_1 a_{k-1},
\end{equation}
where $S_k = -\sum\limits_{i=1}^{2g} \lambda_{i}^k$. Similarly, we have for coefficients $b_k$:
\begin{equation}
\label{eq:b_k_from_power_of_roots}
k b_k = S'_k + S'_{k-1} b_1 + S'_{k-2} b_2 + ... + S'_1 b_{k-1},
\end{equation}
where $S'_k = -\sum\limits_{i=1}^{g} (\lambda_{i}+\frac{q}{\lambda_{i}})^k$.

Now let us consider the situation modulo prime $\ell$ and the restriction of the Frobenius endomophism $\varphi$ on $A[\ell]$.
\begin{proposition}
\label{prop:bk_modulo_ell}
Let $A$ be an abelian variety of dimension $g$ over a finite field $\mathbb{F}_q$ of characteristic $p$, let $h_{A,q}(T) = \sum\limits_{k=0}^{g} b_k T^k$ be the real Weil polynomial of the characteristic polynomial of the Frobenius endomorphism $\varphi$ on $A$. If $\ell \neq p$ is a prime, $r$ is the order of $\varphi$ on $A[\ell]$, and $\gcd(r,\ell) = 1$ then
\[k b_k = S'_k + S'_{k-1} b_1 + S'_{k-2} b_2 + ... + S'_1 b_{k-1}\pmod{\ell}\]
where
\begin{equation*}
\begin{split}
S'_{2k} &= - \sum\limits_{i=1}^{g} (\eta_i q)^k,\\
S'_{2k+1} &= - \sum\limits_{i=1}^{g} (\pm(\eta_i q)^{k+\frac{1}{2}}).
\end{split}
\end{equation*}
Here, $\eta_i = \zeta_i + \frac{1}{\zeta_i} + 2$ for $i = 1,...,g$ and $\zeta_1, ..., \zeta_g$ are some $r$-th roots of unity in $\overline{\mathbb{F}}_\ell$ such that $ \operatorname{lcm}(\operatorname{ord}(\zeta_1), \ldots, \operatorname{ord}(\zeta_g)) = r$ if $r$ is odd and if $r$ is even, $ \operatorname{lcm}(\operatorname{ord}(\zeta_1), \ldots, \operatorname{ord}(\zeta_g)) = r$ or $\frac{r}{2}$.
\end{proposition}
\begin{proof}
Let $F_\ell$ be a matrix representing action of $\varphi$ on $A[\ell]$. So $r$ is the order of $F_\ell$, i.e. a minimal integer $r$ such that $F_\ell^r = \alpha I$ for some $\alpha$.
Let $P_i \in A[\ell]$ be such that $\varphi(P_i) = [\lambda_i] P_i$ and $\tilde{P}_i$ be the corresponding vector from $(\mathbb{Z}/\ell \mathbb{Z})^{2g} \simeq A[\ell]$.
On the one hand, we have $F_\ell^r \tilde{P}_i = \alpha^r$ since $\varphi^r$ is represented by the matrix $F_\ell^r = \alpha I$ for a constant $\alpha$.
On the other hand we have $F_\ell^r \tilde{P}_i = \lambda_i^r \tilde{P}_i$. So $\lambda_{1}^r = \lambda_{2}^r = ... = \lambda_{2g}^r$.
Since $\lambda_i \lambda_{i+g} = q$, we obtain $\lambda_{i}^r \lambda_{i+g}^r = \lambda_{i}^{2r} = q^r$. This implies the relation $\lambda_{i}^2 = \zeta_i q$ for some $r$-th roots of unity $\zeta_i$ and, since $r$ is minimal, we can derive additional restrictions on the $r$-th roots. Let $n = \operatorname{lcm}(\operatorname{ord}(\zeta_1),..., \operatorname{ord}(\zeta_g))$ then
$\lambda_{1}^{2 n} = ... = \lambda_{g}^{2 n} = q^{2 n}$. From this in case of $2 n < r$ we have a contradiction to the minimality of $r$. Hence $2n \geq r$ and, since $n \leq r$ and $n$ is a divisor of $r$, we have $n = r$ or $n = r/2$. 

Let $\eta_i = \zeta_i + \frac{1}{\zeta_i} + 2$ as in the $g=1$ case.
Since $\lambda_i + \frac{q}{\lambda_i} = \frac{(\zeta_i + 1)q}{\lambda_i} = \pm \sqrt{\eta_i q}$, the coefficients $b_k$ of $h_{A,q}(T)$ are elementary symmetric polynomials in variables $\pm \sqrt{\eta_i q}$ .
Using the relation $\lambda_{i}^2 = \zeta_i q$, we can write $S'_{2k}$ and $S'_{2k+1}$ from Eq.~\eqref{eq:b_k_from_power_of_roots} as 
\[
S'_{2k} = - \sum\limits_{i=1}^{g} (\eta_i q)^k\] and
\[
S'_{2k+1} = - \sum\limits_{i=1}^{g} (\eta_i q)^k (\eta_i - 2) \lambda_{i} = - \sum\limits_{i=1}^{g} (\pm(\eta_i q)^{k+\frac{1}{2}}).
\]
\end{proof}
In case of $\gcd(\ell, r) \neq 1$ the Eq.~\eqref{eq:PSp_2g_order} implies that we can write $r$ as $r = \ell^k r_0$ where $k | g^2$ and $\ell \nmid r_0$. In this case we can take $r = r_0$ in Proposition \ref{prop:bk_modulo_ell}.

Finally, we apply Proposition \ref{prop:bk_modulo_ell} to obtain relations modulo $\ell$:
\begin{equation}
\label{eq:b_i_from_zetas}
\begin{split}
b_1 &= -\sum\limits_{i=1}^{g} (\pm\sqrt{\eta_i q}), \\
2 b_2 &= - \sum\limits_{i=1}^{g} (\eta_i q) + b_1^2,\\
3 b_3 &= - \sum\limits_{i=1}^{g} \pm (\eta_i q)^{3/2} + (2 b_2 - b_1^2)b_1 + b_1 b_2, \\
&... \\ 
(2k) b_{2k} &= - \sum\limits_{i=1}^{g} (\eta_i q)^k + f_{2k}(b_1, ..., b_{2k-1}), \\
(2k+1) b_{2k+1} &= - \sum\limits_{i=1}^{g} (\pm(\eta_i q)^{k+\frac{1}{2}}) + f_{2k+1}(b_1,...,b_{2k}),
\end{split}
\end{equation}
where $f_{2k}$ and $f_{2k+1}$ are polynomials obtained by substituting the previously computed values of $S'_i$ to the Eq.~\eqref{eq:b_k_from_power_of_roots}.

Thus, the coefficients $a_k$ can be written in terms of $\sqrt{\eta_i q}$  for $k=1,..., g$ by Eqs.~\eqref{eq:a_from_b_even_case} and \eqref{eq:a_from_b_odd_case}. In the following we also use squaring and the fact that $b_i$ are elementary symmetric polynomials in $\pm \sqrt{\eta_i q}$ to get rid of signs and to make formulae \eqref{eq:b_i_from_zetas} simpler. For example, we can write $b_g^2  = \eta_1 \cdot ... \cdot \eta_g q^g$ for the coefficient $b_g$.

Note that for $g=1$ these formulae give us the formulae from Proposition 6.2 in \cite{Schoof1995}.
For the case $g = 2$ and $g = 3$, we obtain the following propositions.
\begin{proposition}
\label{prop:g2_coeffs_and_roots}
Let $A$ be an abelian surface over a finite field $\mathbb{F}_q$ and $\chi_{A,q}(T) = T^{4} + a_1 T^{3} + a_2 T^2 + a_{1} qT + q^2$ be the characteristic polynomial of the Frobenius endomorphism $\varphi$ on $A$, let $r$ be the order of $\varphi$ on $A[\ell]$ for $\ell \neq p$, and let $\gcd(\ell,r) = 1$ then
    \begin{equation}
    \label{eq:g2_r_zeta_1}
    a_1^2 = (\sqrt{\eta_1} \pm \sqrt{\eta_2})^2 q \pmod{\ell},
    \end{equation}
    and
    \begin{equation}
    \label{eq:g2_r_zeta_2}
    (a_2 - 2q)^2 = \eta_1 \eta_2 q^2 \pmod{\ell}
    \end{equation}
    where $\eta_1 = \zeta_1 + \zeta_1^{-1} + 2$, $\eta_2 = \zeta_2 + \zeta_2^{-1} + 2$ and $\zeta_1, \zeta_2$ are some $r$-th roots of unity such that $\operatorname{lcm}(\operatorname{ord}(\zeta_1), \operatorname{ord}(\zeta_2)) = r$ in case $r$ is odd and in case $r$ is even, $\operatorname{lcm}(\operatorname{ord}(\zeta_1), \operatorname{ord}(\zeta_2)) = r$ or $\frac{r}{2}$.
\end{proposition}
\begin{proof}
By Eq.~\eqref{eq:a_from_b_odd_case}, we have $a_1 = b_1 = (\pm \sqrt{\eta_1} \pm \sqrt{\eta_2}) \sqrt{q}$. Therefore, $a_1^2 = (\sqrt{\eta_1} \pm \sqrt{\eta_2})^2 q$.
Since $b_2 = \pm \sqrt{\eta_1} \sqrt{\eta_2} q$, we have $b_2^2 = \eta_1 \eta_2 q^2$. From Eq.~\eqref{eq:a_from_b_even_case}, we can write $a_2 = b_2 + 2 q$ and therefore $(a_2 - 2q)^2 = \eta_1 \eta_2 q^2$.
\end{proof}
If $\gcd(\ell,r) \neq 1$ then, as in the general case, we can take the integer $r_0$ such that $r = \ell^k r_0$, $\ell \nmid r_0$ and apply the Proposition \ref{prop:g2_coeffs_and_roots} for $r = r_0$.

The formulae in Proposition \ref{prop:g2_coeffs_and_roots} appears in \cite{Martindale2017} with additional restrictions on abelian variety $A$. Our version is fully general with weakened conditions on roots of unity.

\begin{proposition}
\label{prop:g3_coeffs_and_roots}
Let $A$ be an abelian variety of dimension $3$ over a finite field $\mathbb{F}_q$ and $\chi_{A,q}(T) = T^6 + a_1 T^5 + a_2 T^4 + a_3 T^3 + a_2 q T^2 + a_1 q^2 T + q^3$ be the characteristic polynomial of the Frobenius endomorphism $\varphi$ on $A$, let $r$ be the order of $\varphi$ on $A[\ell]$ for $\ell \neq p$, and let $\gcd(\ell,r) = 1$ then
\[
a_1^2 = (\pm \sqrt{\eta_1} \pm \sqrt{\eta_2} \pm \sqrt{\eta_3})^2 q,
\]
\[
2 a_2 = a_1^2 + 6 q - (\eta_1 + \eta_2 + \eta_3),
\]
\[
(a_3 - 2 a_1 q)^2 = \eta_1 \eta_2 \eta_3 q^3
\]
modulo $\ell$, where $\eta_1 = \zeta_1 + \frac{1}{\zeta_1} + 2$, $\eta_2 = \zeta_2 + \frac{1}{\zeta_2} + 2$, $\eta_3 = \zeta_3 + \frac{1}{\zeta_3} + 2$
for some $r$-th roots of unity $\zeta_1, \zeta_2, \zeta_3$ such that
$\operatorname{lcm}(\operatorname{ord}(\zeta_1), \operatorname{ord}(\zeta_2), \operatorname{ord}(\zeta_3)) = r$ in case $r$ is odd and in case $r$ is even, $\operatorname{lcm}(\operatorname{ord}(\zeta_1), \operatorname{ord}(\zeta_2), \operatorname{ord}(\zeta_3)) = r$ or $\frac{r}{2}$.
\end{proposition}
\begin{proof}
1. First relation follows from the fact that $a_1 = b_1 = (\pm \sqrt{\eta_1} \pm \sqrt{\eta_2} \pm \sqrt{\eta_3}) \sqrt{q}$.

2. Since $a_2 = b_2 + 3 q$ by Eq.~\eqref{eq:a_from_b_even_case}, we have $2 a_2 - 6 q = 2 b_2 = - (\eta_1 + \eta_2 + \eta_3) + a_1^2$. 

3. We have $b_3^2 = \eta_1 \eta_2 \eta_3 q^3$. Equations   \eqref{eq:a_from_b_odd_case} and \eqref{eq:a_from_b_even_case} imply $a_3 = b_3 + 2 q b_1 = b_3 + 2 q a_1$. Then $(a_3 - 2 q a_1)^2 = \eta_1 \eta_2 \eta_3 q^3$.
\end{proof}

\section{Conjugacy classes and the orders of elements in $\operatorname{Sp}_4(\mathbb{F}_\ell)$}
\label{sec:conjugacy_classes}
In general case the orders of matrices over a finite field were considered in \cite{Stong1993,Schmutz1995,Fulman2002,Schmutz2008,AivazidisSofos2015}. In this section we study the distribution of orders of matrices in $\operatorname{Sp}_4(\mathbb{F}_\ell)$ as elements of projective symplectic group $\operatorname{PSp}_4(\mathbb{F}_\ell)$.
We define the order of a matrix $M \in \operatorname{Sp}_4(\mathbb{F}_\ell)$ to be the minimal number $r$ such that $M^r = \lambda I$ for some scalar $\lambda \in \mathbb{F}_\ell$.
We need such specific definition to derive the properties of Frobenius orders in the next section.
All similar matrices have the same order, so it is enough to find the orders of conjugacy classes.
A description of conjugate classes in $\operatorname{Sp}_4(\mathbb{F}_\ell)$ with explicit representatives is given in Srinivasan's work \cite[p.~489-491]{Srinivasan1968}. Using the same notation we denote the conjugacy classes in $\operatorname{Sp}_4(\mathbb{F}_\ell)$ by $\overline{A}_\bullet, \overline{B}_\bullet (\bullet), \overline{C}_\bullet (\bullet), \overline{D}_\bullet$ with representative elements $A_\bullet, B_\bullet (\bullet), C_\bullet (\bullet), D_\bullet$ respectively.
For each class we calculate orders $r=\operatorname{ord}(M)$ of matrices by using the explicit representatives. Since the number of matrices in a class is also known, we can calculate the probability of a random matrix $M \in \operatorname{Sp}_4(\mathbb{F}_\ell)$ to fall in a given class. We give the orders for classes with their respective probabilities in Table~\ref{table:orders}
.
\begin{table}[H]
    \caption{\small Orders of matrices in $\operatorname{Sp}_4(\mathbb{F}_\ell)$ as elements of $\operatorname{PSp}_4(\mathbb{F}_\ell)$ and their probabilities.}
    \footnotesize
    \label{table:orders}
    \centering
    \begin{tabular}{|l|l|l|}
        \hline
        Classes in $\operatorname{Sp}_4(\mathbb{F}_\ell)$ & Order of matrices (projective) & Probability ($M \in \operatorname{Sp}_4(\mathbb{F}_\ell) \wedge M \in class)$ \\ \hline
        $\overline{A}_1, \overline{A'}_1$   & $1$ & $1/(\ell^4(\ell^2-1)(\ell^4-1))$ \\ \hline
        $\overline{A}_{21}, \overline{A'}_{21}, \overline{A}_{22}, \overline{A'}_{22}$ & $\ell$ & $1/(2\ell^4(\ell^2-1))$ \\ \hline
        $\overline{A}_{31}, \overline{A'}_{31}$ & $\ell$ & $1/(2\ell^3(\ell-1))$ \\ \hline
        $\overline{A}_{32}, \overline{A'}_{32}$& $\ell$ & $1/(2\ell^3(\ell+1))$ \\ \hline
        $\overline{A}_{41}, \overline{A'}_{41}, \overline{A}_{42}, \overline{A'}_{42}$ & $\ell$ & $1/(2\ell^2)$ \\ \hline
        $\overline{B}_1(i)$ & $\frac{\ell^2+1}{2s}, s=\gcd(i,\frac{\ell^2+1}{2})$ &                        $1/(\ell^2+1)$ \\ \hline
        $\overline{B}_2(i)$ & $\frac{\ell^2-1}{2s}, s=\gcd(i,\frac{\ell^2-1}{2})$ &                        $1/(\ell^2-1)$ \\ \hline
        $\overline{B}_3(i,j)$ & $\frac{\ell-1}{\gcd(\ell-1,i+j, \vert i-j \vert)}$ &                         $1/(\ell-1)^2$ \\ \hline
        $\overline{B}_4(i,j)$ & $\frac{\ell+1}{\gcd(\ell+1,i+j, \vert i-j \vert)}$ &                        $1/(\ell+1)^2$ \\ \hline
        $\overline{B}_5(i,j)$ & $\frac{\ell^2-1}{\gcd(\ell^2-1, i(\ell-1)+j(\ell+1), 2i(\ell-1))}$ &    $1/(\ell^2-1)$            \\ \hline
        $\overline{B}_6(i)$ & $\frac{\ell+1}{2s}, s=\gcd(i,\frac{\ell+1}{2})$ &                        $1/(\ell(\ell+1)(\ell^2-1))$ \\ \hline
        $\overline{B}_7(i)$ & $\frac{\ell(\ell+1)}{2s}, s=\gcd(i,\frac{\ell(\ell+1)}{2})$ &                        $1/(\ell(\ell+1)) $ \\ \hline
        $\overline{B}_8(i)$ & $\frac{\ell-1}{2s}, s=\gcd(i,\frac{\ell-1}{2})$ &                        $1/(\ell(\ell-1)(\ell^2-1))$ \\ \hline
        $\overline{B}_9(i)$ & $\frac{\ell(\ell-1)}{2s},s=\gcd(i,\frac{\ell(\ell-1)}{2})$                                                                           &      $1/(\ell(\ell-1))$   \\ \hline
        $\overline{C}_1(i)$ & $\frac{\ell+1}{s}, s=\gcd(i,\ell+1)$ &                         $1/(\ell(\ell+1)(\ell^2-1))$ \\ \hline
        $ \overline{C'}_1(i)$ & 
        \vtop{\hbox{\strut 
                $\begin{cases} \begin{array}{ll} 2s, & \text{if } 2 \nmid s, \text{and } 4 \nmid s \\ \frac{s}{2}, & \text{if } 2 \mid s, \text{and } 4\nmid s \\ s, & \text{if } 4 \mid s \end{array} \end{cases}$
            }\hbox{\strut 
                $\text{where } s=\frac{\ell+1}{\gcd(i,\ell+1)}$
        }}
        &                         $1/(\ell(\ell+1)(\ell^2-1))$ \\ \hline
        $\overline{C}_{21}(i), \overline{C}_{22}(i)$ & $\frac{\ell(\ell+1)}{s}, s=\gcd(i,\ell(\ell+1))$ &                         $1/(2\ell(\ell+1))$ \\ \hline
        $\overline{C'}_{21}(i), \overline{C'}_{22}(i)$ & 
        \vtop{\hbox{\strut 
                $\begin{cases} \begin{array}{ll} 2s, & \text{if } 2 \nmid s, \text{and } 4 \nmid s \\ \frac{s}{2}, & \text{if } 2 \mid s, \text{and } 4\nmid s \\ s, & \text{if } 4 \mid s \end{array} \end{cases}$
            }\hbox{\strut 
                $\text{where } s=\frac{\ell(\ell+1)}{\gcd(i,\ell(\ell+1))}$
        }}
        &                         $1/(2\ell(\ell+1))$ \\ \hline
        $\overline{C}_{3}(i)$ & $\frac{\ell-1}{\gcd(i,\ell-1)}$ &                         $1/\ell(\ell-1)(\ell^2-1)$ \\ \hline
        $\overline{C'}_{3}(i)$ & 
        \vtop{\hbox{\strut 
                $\begin{cases} \begin{array}{ll} 2s, & \text{if } 2 \nmid s, \text{and } 4 \nmid s \\ \frac{s}{2}, & \text{if } 2 \mid s, \text{and } 4\nmid s \\ s, & \text{if } 4 \mid s \end{array} \end{cases}$
            }\hbox{\strut 
                $\text{where } s=\frac{\ell-1}{\gcd(i,\ell-1)}$
        }}
        &                         $1/\ell(\ell-1)(\ell^2-1)$ \\ \hline
        $\overline{C}_{41}(i), \overline{C}_{42}(i)$ & $\frac{\ell(\ell-1)}{\gcd(i,\ell(\ell-1))}$ & $1/(2\ell(\ell-1))$ \\ \hline
        $\overline{C'}_{41}(i), \overline{C'}_{42}(i)$ & 
        \vtop{\hbox{\strut 
                $\begin{cases} \begin{array}{ll} 2s, & \text{if } 2 \nmid s, \text{and } 4 \nmid s \\ \frac{s}{2}, & \text{if } 2 \mid s, \text{and } 4\nmid s \\ s, & \text{if } 4 \mid s \end{array} \end{cases}$
            }\hbox{\strut 
                $\text{where } s=\frac{\ell(\ell-1)}{\gcd(i,\ell(\ell-1))}$
        }}
        & $1/(2\ell(\ell-1))$ \\ \hline
        
        $\overline{D}_1$ & $2$ & $1/(\ell^2(\ell^2-1)^2)$ \\ \hline
        $\overline{D}_{21}, \overline{D}_{22}, \overline{D}_{23}, \overline{D}_{24}$ & $2\ell$ &  $1/(2\ell^2(\ell^2-1))$  \\ \hline
        $\overline{D}_{31}, \overline{D}_{32}, \overline{D}_{33}, \overline{D}_{34}$ & $2\ell$ &  $1/(4\ell^2)$ \\ \hline
    \end{tabular}
\end{table}

Having explicit information on orders of matrices in classes, we can now derive numerical characteristics of the distribution of orders. Let $\xi$ be a random variable that takes values in $\lbrace \operatorname{ord}(M) | M \in \operatorname{Sp}_4(\mathbb{F}_\ell) \rbrace $.
Our next goal is to find an expected value and variance of the random variable $\xi$. Define the expected order of a matrix in $\operatorname{Sp}_4(\mathbb{F}_\ell)$ as $\mu_4 = \frac{1}{\#\operatorname{Sp}_4(\mathbb{F}_\ell)} \sum_{M \in \operatorname{Sp}_4(\mathbb{F}_\ell)} \operatorname{ord}(M)$, where the order is defined for $M$ as an element of $\operatorname{PSp}_4(\mathbb{F}_\ell)$.

Since all matrices in a conjugacy class have the same order, we can split the sum $\mu_4$ into parts which correspond to the conjugacy classes. For a conjugacy class $\overline{M}$ the corresponding term in the sum $\mu_4$ is given by the formula
\[
\mu(\overline{M})=\operatorname{ord}(M)\cdot\frac{\#\overline{M}}{\#\operatorname{Sp}_{4}(\mathbb{F}_\ell)}=\operatorname{ord}(M)\cdot \operatorname{Pr}(\xi=\operatorname{ord}(M)).
\]
For classes $\overline{A},\overline{D}$ the order is fixed. For classes of type $\overline{B}_k(i,j)$, $\overline{B}_k(i)$, $\overline{C}_k(i)$, $\overline{C'}_k(i)$ the order depends on parameters $i,j$ and we assume that parameters $i,j$ are distributed uniformly among their value sets as $\ell \rightarrow \infty$. So, we can use the following approximation \cite{DiaconisErdos2004} for $\gcd(i, x)$:
\[
E(x):=\frac{6}{\pi^2}\log(x)+O(1).
\]
The expected orders of symplectic matrices in $\cup_{i,j}\overline{B}_k(i,j)$, $\cup_i \overline{B}_k(i)$, $\cup_i \overline{C}_k(i)$, $\cup_i \overline{C'}_k(i)$ are presented in Table \ref{table:conj_classes_exp_orders}
.
\begin{table}[H]
    \centering
    \caption{\small Expected orders of symplectic matrices $\operatorname{Sp}_4(\mathbb{F}_\ell)$ as elements of $\operatorname{PSp}_4(\mathbb{F}_\ell)$.}
    \label{table:conj_classes_exp_orders}
    \begin{tabular}{|l|l|l|}
        \hline
        \multicolumn{1}{|c|}{Classes}                                                             & Quantity of $i,j$ & \multicolumn{1}{c|}{Expected order}        \\ \hline
        $\overline{B}_1(i)$                                                                                 & $\frac{1}{4}(\ell^2-1)$ & $\frac{1}{2}E^{-1}(\frac{\ell^2+1}{2})$        \\ \hline
        $\overline{B}_2(i)$                                                                                 & $\frac{1}{4}(\ell-1)^2$ & $\frac{1}{2}E^{-1}(\frac{\ell^2-1}{2})$        \\ \hline
        $\overline{B}_3(i,j)$                                                                               & $\frac{1}{8}(\ell-3)(\ell-5)$ & $\frac{1}{\ell-1}E^{-1}(\ell-1)$          \\ \hline
        $\overline{B}_4(i,j)$                                                                               & $\frac{1}{8}(\ell-1)(\ell-3)$ & $\frac{1}{\ell+1}E^{-1}(\ell+1)$          \\ \hline
        $\overline{B}_5(i,j)$                                                                               & $\frac{1}{4}(\ell-1)(\ell-3)$ & $E^{-1}(\ell^2-1)$             \\ \hline
        $\overline{B}_6(i)$                                                                                 & $\frac{1}{2}(\ell-1)$ & $\frac{1}{2\ell(\ell^2-1)}E^{-1}(\frac{\ell+1}{2})$        \\ \hline
        $\overline{B}_7(i)$                                                                                 & $\frac{1}{2}(\ell-1)$ & $\frac{1}{2}E^{-1}(\frac{\ell(\ell+1)}{2})$            \\ \hline
        $\overline{B}_8(i)$                                                                                 & $\frac{1}{2}(\ell-3)$ & $\frac{1}{2\ell(\ell^2-1)}E^{-1}(\frac{\ell-1}{2})$ \\ \hline
        $\overline{B}_9(i)$                                                                                 & $\frac{1}{2}(\ell-3)$ & $\frac{1}{2}E^{-1}(\frac{\ell(\ell-1)}{2})$         \\ \hline
        $\overline{C}_1(i),\overline{C'}_1(i)$                                                                         & $(l-1)$ & $\frac{1}{\ell(\ell^2-1)}E^{-1}(\ell+1)$              \\ \hline
        \begin{tabular}[c]{@{}l@{}}$\overline{C}_{21}(i),\overline{C'}_{21}(i),$\\ $\overline{C}_{22}(i),\overline{C'}_{22}(i)$\end{tabular} & $2(l-1)$ & $\frac{1}{2}E^{-1}(\ell+1)$                      \\ \hline
        $\overline{C}_3(i),\overline{C'}_3(i)$                                                                         & $(l-3)$ & $\frac{1}{\ell(\ell^2-1)}E^{-1}(\ell-1)$              \\ \hline
        \begin{tabular}[c]{@{}l@{}}$\overline{C}_{41}(i),\overline{C'}_{41}(i),$\\ $\overline{C}_{42}(i),\overline{C'}_{42}(i)$\end{tabular} & $2(l-3)$ & $\frac{1}{2}E^{-1}(\ell-1)$                      \\ \hline
    \end{tabular}
\end{table}
Now, we obtain the expected order and the variance of the symplectic matrix applying the well-known formulae to the data from this table.

\begin{proposition}
\label{prop:Sp4_expected_order_and_variance}
Let $M$ be a matrix from $\operatorname{Sp}_4(\mathbb{F}_\ell)$. Define the order of $M$ to be the order of $M$ in the group $\operatorname{PSp}_4(\mathbb{F}_\ell)$. Then
\begin{enumerate}
    \item The expected order of matrix $M$ is equal to
    \[   
    \mu_4=
    \frac{\pi^2}{48\ell(\ell^2-1)}\cdot (2\ell^5+15\ell^4-47\ell^3+\ell^2+65\ell-40)\log^{-1}(\ell).
    \]
    \item The variance of the order's distribution is equal to
    \[   
    \delta_4=
    \left(\frac{\pi}{24\ell(\ell^2-1)}\right)^2\cdot \psi (\ell) \log^{-1}(\ell) - \mu_4^2,
    \]
    where
    \[
    \psi (\ell) = 6\ell^{10} - 27\ell^9 + 420\ell^8 - 1443\ell^7 + 828\ell^6 + 3375\ell^5 - 3804\ell^4 - 825\ell^3 + 2550\ell^2 - 1080\ell.
    \]
\end{enumerate}
\end{proposition}

\section{Distribution of orders of the Frobenius action on $A[\ell]$}
\label{sec:order}
Let $A$ be an abelian surface defined over a finite field $\mathbb{F}_q$ of odd characteristic $p$. From \S\ref{sec:frob_action} we know that the action of the Frobenius endomorphism on $\ell$-torsion subgroup in case $q \equiv 1 \pmod{\ell} $ is represented by a symplectic matrix from $ \operatorname{Sp}_{4}(\mathbb{F}_\ell)$. To find the distribution of the Frobenius orders, we use a heuristic assumption that the elements of Frobenius are equidistributed in $\operatorname{Sp}_4(\mathbb{F}_\ell)$.
The assumption was already used in \cite{AchterWilliams2015} in the context of counting the number of isogeny classes of abelian varieties.
Thus, from Proposition \ref{prop:Sp4_expected_order_and_variance} we obtain our results for expected order and variance of the Frobenius order.

\begin{theorem}
\label{th:expected_order}
Let $A$ be an abelian surface defined over a finite field $\mathbb{F}_q$ of characteristic $p$. If $\ell \neq p$ is a prime number and $q \gg \ell$, then the expected order of the Frobenius action on $A[\ell]$ is equal to $\mu_4$.
\end{theorem}

\begin{theorem}
Let $A$ be an abelian surface defined over a finite field $\mathbb{F}_q$ of characteristic $p$. If $\ell \neq p$ is a prime number, $q \gg \ell$, then the variance of order distribution of the Frobenius action on $A[\ell]$ is equal to $\delta_4$.
\end{theorem}

\begin{theorem}
(Heuristic). The modes of the random variable $\xi$ are $\frac{\ell^2+1}{2}$ and $\frac{\ell^2-1}{2}$.
\end{theorem}
In point counting algorithms, we have to enumerate all primes $\ell \leq (9g+3) \log{q}$. For cryptography on genus $2$ curves we work with fields of size $160$-bit. In this case the size of the group will be equal to $O(q^2)$ by the Hasse-Weil bound, i.e. $320$-bit. So we have to find the characteristic polynomial $\chi_{A,q}(T) \pmod{\ell}$ for all primes $\ell \leq 3360$ to restore the coefficients of $\chi_{A,q}(T)$ by CRT.

Using the data from Table \ref{table:orders}, we calculated the distribution of the Frobenius orders for the first $500$ primes $\ell = 3 \ldots 3571$. Since the order of any matrix depends linearly on $\ell^2$, as follows from Theorem \ref{th:expected_order}, we normalize the order by calculating the value $\frac{\operatorname{ord}(M)}{\ell^2}$ instead of $\operatorname{ord}(M)$ itself. An obtained family of distributions is shown on Fig. \ref{fig:heatmap}. Taking an average value of order on different $\ell$'s, one can construct the averaged distribution of orders. We present this distribution in the Table~\ref{table:distribution}.

\begin{table}[H]
    \centering
	\caption {The distribution for orders of the Frobenius action on $A[\ell]$.}
	\label{table:distribution}
\footnotesize{
	\tabcolsep=0.08cm
\begin{tabular}{|c|c|c|c|c|c|c|c|c|c|c|c|c|}
	\hline
	\multirow{2}{*}{Order} & \multicolumn{4}{c|}{$(1,\ell]$} & \multicolumn{2}{c|}{$(\ell,2\ell]$} &  \multicolumn{5}{c|}{$(2\ell,\frac{\ell^2+1}{2}]$} & \multirow{2}{*}{$(\frac{\ell^2+1}{2}, \ell(\ell+1)]$} \\ \cline{2-12} 
	& $\frac{\ell-1}{2}$    & $\frac{\ell+1}{2}$    & $\ell-1$   & Other   & $\ell+1$          & Other          & $\frac{\ell^2-1}{4}$   & $\frac{\ell^2+1}{4}$  & $\frac{\ell^2-1}{2}$  & $\frac{\ell^2+1}{2}$  & Other &  \\ \hline
	\%                     & $4.0$    & $4.0$    & $5.0$   & $6.3$   & $5.0$          & $1.5$          & $6.6$   & $5.0$  & $13.4$  & $15.7$  & $29.7$ & $3.8$  \\ \hline
\end{tabular}
}
\end{table}

\begin{figure}[H]
    \includegraphics[width=\linewidth]{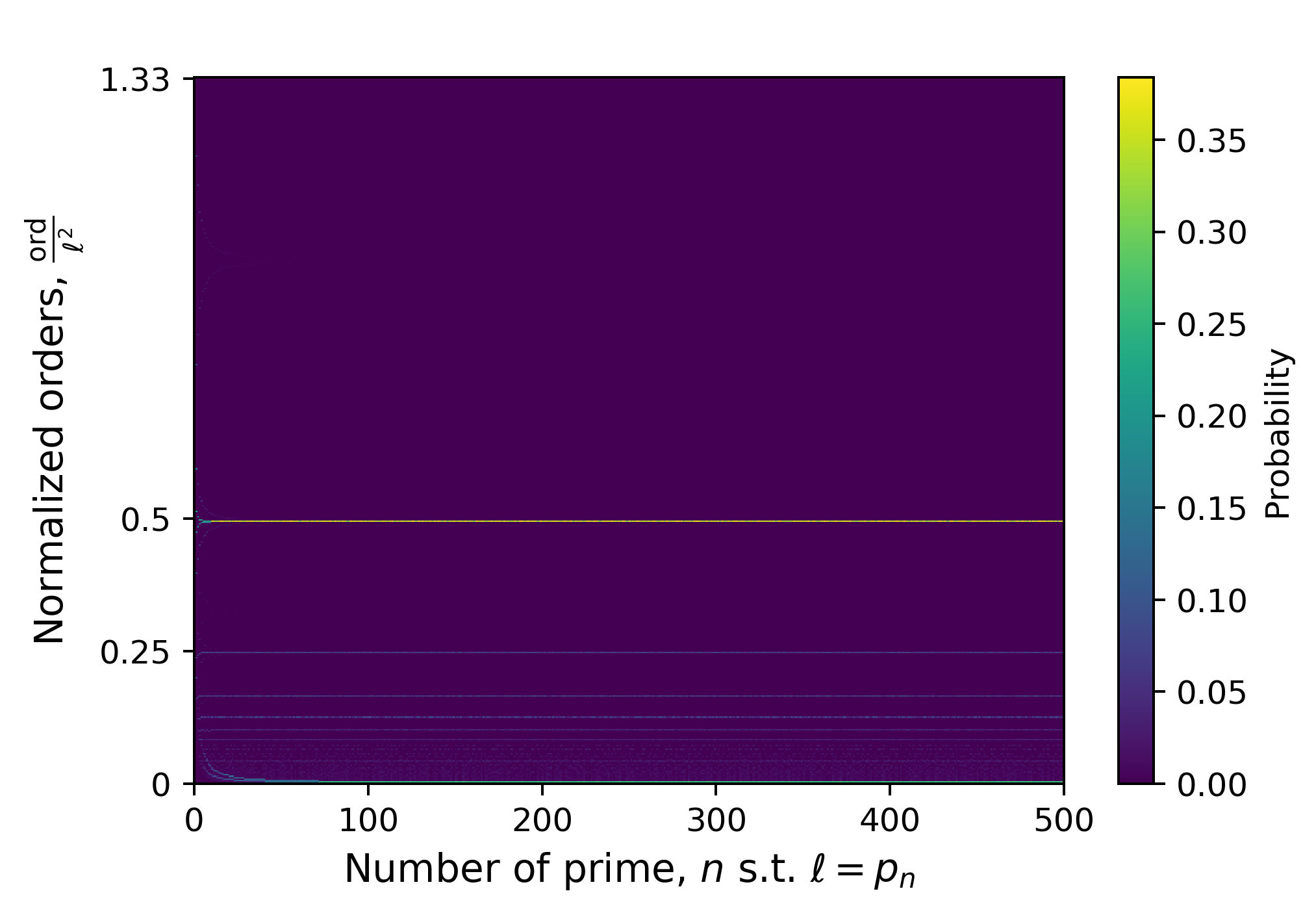}
    \caption{A distribution of orders for the first 500 primes $\ell$.}
    \label{fig:heatmap}
\end{figure}

\section{Application to point counting}
\label{sec:point_counting}
In Schoof-Pila \cite{Pila1990} algorithm determining $\chi_{A,q}(T) \pmod{\ell}$ is done by direct enumeration of at most $\ell^g$ possible coefficients. Each test requires expensive operations like ideal membership test and operations with division polynomials of degree $\ell^{2g}$. So reducing the number of elements to enumerate is crucial.

The obtained results can be used for point counting on abelian variety $A$ in the following modification of Schoof-Pila method.
\begin{enumerate}
\item Choose primes $\ell$ such that $\prod\limits_{\ell \leq H \log{q}} \ell > 2 \binom{2g}{g} q^g$, where $H = (9g+3)$.
\item For each prime $\ell$:
\begin{enumerate}
    \item Build a list $L$ of tuples $(a_1, \ldots, a_g, w)$, where $a_1, \ldots, a_g \pmod{\ell}$ are the candidates for coefficients of characteristic polynomial $\chi_{A,q}(T) \pmod \ell$ and $w$ is a probability of $(a_1, \ldots, a_g)$ to be the coefficients of $\chi_{A,q}(T) \pmod \ell$. This probability is computed by using the distribution of the orders and formulae \eqref{eq:a_from_b_even_case}, \eqref{eq:a_from_b_odd_case}, \eqref{eq:b_i_from_zetas}.
    \item Sort the list $L$ by $w$.
    \item
    \label{prob_alg:step:determine_chi_mod_l}
    Determine $\chi_{A,q}(T)$ by testing tuples from the list starting with the ones having high values of $w$.
\end{enumerate}
    \item
    \label{prob_alg:step:determine_chi_by_CRT}
    Determine $\chi_A(T)$ from the list of $\chi_{A,q}(T) \pmod{\ell}$ using CRT.
\end{enumerate}

To test applicability of the distribution to point counting using the method described above, we run a series of experiments in SageMath \cite{sagemath} system. We choose a set of random primes $p$ of size $p > 2^{16}$. For each prime $p$ we compute a set of $\ell \geq 5$ such that $p \equiv 1 \pmod{\ell}$ and a set of $10000$ random genus $2$ hyperelliptic curves with imaginary model
\[
y^2 = f(x) = x^{5} + f_4 x^4 + f_3 x^3 + f_2 x^2 + f_1 x + f_0.
\]
This model is most common in cryptography. Such a curve is generated by a random monic square-free polynomial $f(x)$ in $\mathbb{F}_p[x]$ of degree $5$. By Prop. \ref{prop:g2_coeffs_and_roots}, for each pair $(p,\ell)$ we build a list $L$ of pairs $(a_1^2, (a_2-2q)^2)$ corresponding to small orders $r = \frac{\ell\pm 1}{2}$ from Table \ref{table:distribution}. We choose these  orders because they appear in many conjugacy classes from Table \ref{table:orders} and the most common orders $\frac{\ell^2 \pm 1}{2}$ lead to big lists. For each curve we compute the characteristic polynomial $\chi_p(T)$ of the Frobenius endomorphism by built-in methods of SageMath and so we know the exact value of $\chi_p(T) \pmod{\ell}$. After that we compared the number of attempts to find $\chi_p(T) \pmod{\ell}$ using classical enumeration (as in Schoof-Pila algorithm) against our proposed search in the list $L$.

Our experiments show that the number of attempts to find the $\chi_p(T) \pmod{\ell}$ is reduced by $\approx 1-12\%$ for $\ell \leq 100$ where the success rate is decreasing with the growing of $\ell$. In the case of $\ell > 100$ we have the number of attempts reduced by $\approx 1-2\%$.

To improve this we should generalize Atkin's ``Match and Sort" algorithm for elliptic curves and use the data on Frobenius distribution in this generalized algorithm. This can be done by combining steps \ref{prob_alg:step:determine_chi_mod_l} and \ref{prob_alg:step:determine_chi_by_CRT} in the algorithm above and by using baby-step giant-step algorithm to determine $\chi_{A,q}(T)$ from the lists $L$ each corresponding to different $\ell$. However a realization of this method is still an open problem even in the case when we use modular polynomials to determine the Frobenius order and so we know the exact order. 

\section{Conclusion}
In this work we presented a generalization of Atkin's formulae to any dimension and showed that the distribution of Frobenius orders is not uniform for abelian surfaces over a finite field $\mathbb{F}_q$ with $q \equiv 1 \pmod{\ell}$. Furthermore, we described possible applications of this distribution to point counting purposes.
The formulae can be used to limit the number of possible characteristic  polynomials $\chi_{A,q}(T) \pmod{\ell}$ in case when we know the Frobenius order. The distribution allows us to sort the lists of possible  $\chi_{A,q}(T) \pmod{\ell}$ by probability.

The further work is to use this modular information about distribution efficiently in the generalization of Schoof's algorithm for genus $2$ curves \cite{GaudrySchost2012}. For elliptic curves there exist Atkin's ``Match and Sort'' algorithm and ``Chinese and Match'' algorithm \cite{JouxLercier2001} due to Joux and Lercier. But for higher dimension this is still an open problem.
\bibliography{article.bib}

\end{document}